\documentclass[11pt,leqno,twoside]{amsart}

\usepackage{enumerate}

\usepackage{amssymb}

\usepackage{amssymb,amsthm,amsmath,eucal,mathrsfs,verbatim}

\usepackage{footnote}

\usepackage{tikz, xcolor}

\usepackage{cite}

\usepackage{amsmath}

\usepackage{amsthm}

\usepackage{amssymb}

\usepackage{amsfonts}

\usepackage{hyperref}

\newtheorem{theorem}{Theorem}[section]

\newtheorem{lemma}[theorem]{Lemma}

\newtheorem{proposition}[theorem]{Proposition}

\theoremstyle{definition}

\newtheorem{definition}[theorem]{Definition}

\newtheorem{remark}[theorem]{Remark}

\numberwithin{equation}{section}

\newcommand\restr[2]{{
  \left.\kern-\nulldelimiterspace 
  #1 
  \vphantom{\big|} 
  \right|_{#2} 
  }}

\usepackage{eucal}

\title[Transport Equation Coupled with a Balance Law]{Existence and Uniqueness of a Transport Equation coupled  with a Balance law}
\author[Banerjee and Sahoo\hfil\hfilneg]{Shuvam Banerjee and Manas R. Sahoo}
\email{}

\keywords{Balance laws; Hamilton Jacobi equation; Explicit formula; $h$-curves; Lax-{O}le\u\i nik formula; Characteristic triangles; Generalised characteristics.\\
School of Mathematical Sciences, National Institute of Science Education and Research, An OCC of Homi Bhabha National Institute, Bhubaneswar, P.O. Jatni, Khurda, Odisha 752050, India. Email: shuvam.banerjee@niser.ac.in, manas@niser.ac.in}

\begin{document}

\maketitle 

\begin{abstract}

In this article, we study the existence and uniqueness of a transport equation with discontinuous coefficients that comes from the solution of a balance law with a time-dependent source term. We also give a simplified proof providing an explicit formula for the solution of the balance law using a variational formulation.   This system can be considered as a generalized version of the one-dimensional model for the large-scale structure of the formation of the universe.
\end{abstract}

\section{Introduction}
The multidimensional zero-pressure gas dynamics model reads 
\begin{equation}
\label{multidimensional system}
\begin{aligned}
&u_t+(u.\nabla)u=0,\,\, u=\nabla \phi\\
&\rho_t +\nabla.(\rho u)=0,
\end{aligned}
\end{equation}
where $u$ is the velocity and $\rho$ is the density of the particles. This model describes the formation of the large-scale structure of the universe when pressure is negligible, see Weinberg and Gunn \cite{WeinbergGunn1990} and Zeldovich\cite{Zeldovich1970}. 

For the one-dimensional case, the system \eqref{multidimensional system} becomes
\begin{equation}\label{one dimesional model}
\begin{aligned}
\begin{cases}
    &u_t+(\frac{u^2}{2})_x=0,\,\,\,\,\,\, x\in \mathbb{R},\,\, t>0,\\
    &\rho_t +(u\rho)_x=0,\,\,\,\,\,\, x\in \mathbb{R},\,\, t>0.\\
    \end{cases}
    \end{aligned}
\end{equation}
The system \eqref {one dimesional model} is a well-studied equation, see \cite{Joseph2008,Joseph1993, JosephSachdev2001}. 
 
We are interested in studying the initial value problem of the following system in one space dimension,
\begin{equation}\label{equation 1}
\begin{aligned}
\begin{cases}
    &u_t+f(u)_x=\alpha(t)u,\,\,\,\,\,\, x\in \mathbb{R},\,\, t>0,\\
    &\rho_t + (f^{\prime}(u) \rho)_x=0,\,\,\,\,\,\, x\in \mathbb{R},\,\, t>0.\\
    \end{cases}
    \end{aligned}
\end{equation}

with initial conditions
\begin{equation}\label{equation 2}
\begin{aligned}
\begin{cases}
    &u(x, 0)=u_{0}(x),\\
    &\rho(x, 0)=\rho_{0}(x),
    \end{cases}
    \end{aligned}
\end{equation}

where $f$ is a $C^2 (\mathbb{R}; \mathbb{R})$ strictly convex function with superlinear growth, i.e. $$\displaystyle{\lim_{|u| \to \infty} }\frac{f(u)}{|u|}=\infty,$$ $\alpha\in \mathbb{L}^{\infty}([0, \infty);\mathbb{R})$, $u_0$ and $\rho_0$ are bounded measurable functions.
For $\alpha(t)=0$ and $f(u)=\frac{u^2}{2}$, equation \eqref{equation 1} becomes \eqref{one dimesional model}.

If we take $\rho=R_x,$ then $(u, R)$ satisfies 

\begin{equation}\label{equation for R}
\begin{aligned}
\begin{cases}
    &u_t+f(u)_x=\alpha(t)u,\,\,\,\,\,\, x\in \mathbb{R},\,\, t>0,\\
    &R_t + f^{\prime}(u) R_x=0,\,\,\,\,\,\, x\in \mathbb{R},\,\, t>0.\\
    \end{cases}
    \end{aligned}
\end{equation}

with initial conditions
\begin{equation}\label{initial condition for u and R}
\begin{aligned}
\begin{cases}
    &u(x, 0)=u_{0}(x),\\
    &R(x, 0)=\int_0 ^{x}\rho_{0}(y) dy.
    \end{cases}
    \end{aligned}
\end{equation}

Since the component $\rho$ of \eqref{equation 1} is usually a measure, we study the existence and uniqueness of the system \eqref{equation for R} with the initial condition \eqref{initial condition for u and R}. Then the solution of the system  \eqref{equation 1} is $(u, \rho),$ where $\rho=R_x$ is understood as the derivative in the sense of distribution.

The main objective of this paper is to study the existence and uniqueness of solutions for the second equation in the system \eqref{equation for R}. We also provide a simplified proof of the existence of solutions for the first equation of the system \eqref{equation for R}. The second equation is a transport equation with discontinuous coefficients, which is itself a challenging topic. For the one-dimensional case, we refer to \cite{BouchutJames1998}, while for the multidimensional setting, we refer to \cite{BouchutJamesMancini2005}.

LeFloch \cite{LeFloch1990} studied the existence and uniqueness of the system \eqref{equation for R} for the case $\alpha(t)=0$ using the Volpert product. The first equation of the system \eqref{equation 1} was studied by Adimurthi et al. \cite{manish}. They studied the existence, uniqueness and structure of the solution for the 1st equation of the system \eqref{equation 1}.

More precisely, in this paper, we obtain the following results.
\begin{enumerate}
    \item We derived the variational formulation for the first equation of \eqref{equation for R} directly from the classical solution using the properties of the Legendre transform.
    \item We give a short proof of the existence of a weak solution to the first equation in \eqref{equation for R} using a modified representation of the potential. This representation involves a function that behaves as a slope in the homogeneous Hamilton-Jacobi equation. Here, we avoided the dynamic programming principle, an important step toward obtaining a weak derivative of the minimized potential. This makes our proof different from the approach used in \cite{manish}.
    \item For the second equation in \eqref{equation for R} we define the weak formulation using the Lebesgue-Stieltjes measure. We showed the existence of solution by introducing intermediate potentials and considering their Radon-Nikodym derivatives,\cite{Wang01, WHD97}.
\item We proved the uniqueness for the system\eqref{equation for R} under the following assumptions.
   \begin{itemize}
    \item $u_0$ and  $\rho_0$ are bounded measurable functions.
    \item $(u, R)\in L^{\infty}_{loc}((0, \infty), BV_{loc}(\mathbb{R}, \mathbb{R}^2))$ 
    \item For any $\phi \in C_c ^{\infty} (\mathbb{R} \times (0,\infty))$, the map $t\to \int R\phi dx$ is continuous.
   \end{itemize} 
\end{enumerate}

This paper is organized as follows. In section $ 2$, we recall some definitions and important properties that are essential for the rest of the paper. In section $3$, we derive the variational formulation of the problem by assuming that the solutions are classical. Then in section $4$ we show that a variational formulation satisfies the system \eqref {equation 1} in a weak sense. In section $5$ we prove the uniqueness of the solution of the system.
\section{Definitions and properties}

Recall some of the definitions from\cite{manish} modified for our use as follows:
\begin{definition}\label{Legendre Transform}
Given a convex function $f:\mathbb{R}\to \mathbb{R}$ with property $\displaystyle{\lim_{|v|\to \infty}}\frac{f(v)}{|v|}=+\infty$ (super-linear growth), the Legendre transform of $f$ is defined as
$$f^*(p):=\sup_{v\in \mathbb{R}}\{p.v-f(v)\} \, \, (p\in \mathbb{R}).$$
\end{definition}

\begin{definition}[{\em h-curve}]\label{h-function} The curve joining $(x, t)$ and $(y,0),$ where $x, y\in \mathbb{R},$ $t>0$ and satisfying the differential equation $X^{\prime}(\theta)=f^{\prime}(y_0e^{\beta(\theta)})$ ( $\beta(t):=\int_0^t\alpha(s)ds$ , $\alpha:[0,\infty)\to \mathbb{R}$ from 
\eqref{equation for R})for some  unique point $y_0$ is called a {\em h-curve.}
\end{definition}

\begin{definition}[{\em h-function}]\label{h-function1}
 We define $h-$function, $h: \mathbb{R} \times (0, \infty)\to \mathbb{R}$  as the value $y_0$ in the above definition\eqref{h-function} with $y=0.$
\end{definition} 

We now record some useful properties of the $h-$ function whose proof is the direct consequence of the definitions\eqref{h-function} and \eqref{h-function1}.

\begin{proposition}[{\em properties of h-function}] \label{properties of h-function}
Let $X(t)$ be the h-curve as in the definition\eqref{h-function}. Then the following statements hold.
\begin{enumerate}
    \item The value $y_0$ of the h-curve\eqref{h-function} joining $(x,t)$ to $(y,0)$ is $h(x-y, t).$
    \item The value of the h-function is constant along the $h$-curve, $h(X(\theta), \theta)= h(x,t)$ for $0<\theta \leq t.$
    \item $h(x, t)$ is a strictly increasing function in the variable $x$ for fixed $t>0$.
    \item $\displaystyle{\lim_{|x|\to \infty} } \frac{\int_0 ^{x} h(y,t) dy}{x}=\infty$
\end{enumerate}   
\end{proposition}

\begin{definition}
\label{weak formulation of R}
    A function $R\in L^{\infty}_{loc}(\mathbb{R}\times (0, \infty))$ with $R(., t) \in BV_{loc}(\mathbb{R})$ for each $t>0$, is said to be the solution of the second equation of  \eqref{equation for R} with the initial condition $R(x,0)=\int_{0}^{x} \rho_0(z) dz$ if the following hold.
\begin{equation*}
 \int_0^\infty\int_{-\infty}^\infty R\phi_tdtdx- \int_0^\infty\int_{-\infty}^\infty f'(u(x,t))\phi(x,t)dR(x)dt=0,   
\end{equation*}
for all $\phi\in C_c^{\infty} (\mathbb{R} \times (0, \infty))$ and $\|R(.,t)-R_0 \|_{\infty} \to 0$ as $t\to 0.$
Here, the integral with respect to \(dR\) is understood in the Lebesgue--Stieltjes sense.
\end{definition}

\section{explicit formula}
In this section, we derive an explicit formula for the solution of \eqref{equation 1}, assuming that the solution is the classical solution. This gives a justification for why such a variational approach is chosen in \cite{manish}. Later, we show that the same explicit formula is a weak entropy solution for bounded measurable initial data $u_0$.
\begin{theorem}\label{classicalsolution}
Let $u, \rho$ be the classical solution of \eqref{equation 1}. For $x, y \in \mathbb{R}$ and $t>0,$ define
\begin{equation*}
F(x, y, t)=\int_{0} ^t f^*(f'(h(x-y, t)e^{\beta (\theta)}))e^{-\beta(\theta)}d\theta+\int_0^{y}u_0(z) dz,
\end{equation*} where $f^*$ is the Legendre transform of $f$ (See definition \eqref{Legendre Transform}).
Suppose that for $x\in \mathbb{R},~~ t>0,$  $\displaystyle{\min_{y\in \mathbb{R}}} F(x,y, t)$ is attained  at a unique point $y(x,t).$ Then the explicit formulae for $u$ and $\rho$ are given by 
$$u(x, t)= h(x-y(x, t), t)e^{\beta(t)},\,\,\, \rho(x, t)= \rho_0 (y(x,t)).$$
\end{theorem}
\begin{proof}
First, observe that if $U(x, t)$ satisfies
\begin{equation}\label{equation 3.1}
\begin{aligned}
   U_t + f(U_x)&= \alpha(t)U\\
   U(x, 0)&=\int_0^{x}u_0(z) dz,
\end{aligned}
\end{equation}
then $u(x, t)=U_x (x, t)$ satisfies \eqref{equation 1}. First, we obtain the explicit expression for $U$. To do this, 
we rewrite the equation\eqref{equation 3.1} as follows.
\begin{equation*}
\begin{aligned}
    U_t &= -f(U_x)+ \alpha(t)U\\
        &= -\displaystyle{\sup_{p\in \mathbb{R}}}\big(p. U_x + f^{*}(p)\big)+ \alpha(t)U
        &\leq -p.U_x - f^{*}(p)+ \alpha(t)U,\,\, \textnormal{for any $p\in \mathbb{R}$.}
    \end{aligned}
\end{equation*}
Multiplying the above inequality by $e^{-\beta(t)},$ we get the following after simplification.
$$[e^{-\beta(t)} U(y,t)]_t + p(t) [e^{-\beta(t)} U(y,t)]_y \leq e^{-\beta(t)}f^{*}(p(t)) ,\,\, \textnormal{for any $p=p(t)\in \mathbb{R}$ and $(y, t)\in \mathbb{R}\times [0, \infty)$}.$$
Now choose $p(t)= X^{\prime}(t)= f^{\prime}(q e^{\beta(t)}),$ where $q\in \mathbb{R}$ is arbitrary and replace $y$ by $y+X(t)$ to get
\begin{equation}\label{equation 3.2}
[e^{-\beta(t)} U]_t (y+X(t), t) + f^{\prime}(q e^{\beta(t)}) [e^{-\beta(t)} U]_x (y+X(t), t)\leq e^{-\beta(t)}f^{*}(f^{\prime}(q e^{\beta(t)})),
\end{equation}
 where $q\in \mathbb{R}$ is arbitrary. The above equation\eqref{equation 3.2} can be simplified as
\begin{equation}\label{equation 3.3}
\frac{d}{dt}\Big[e^{-\beta(t)} U (y+X(t), t)\Big]\leq e^{-\beta(t)}f^{*}(f^{\prime}(q e^{\beta(t)})).
\end{equation}
Integrating the above equation\eqref{equation 3.3} over $[0, t]$  yields
$$e^{-\beta(t)} U (y+X(t), t)\leq U(y, 0))+\int_0 ^t e^{-\beta(s)}f^{*}(f^{\prime}(q e^{\beta(s)}))\, ds.$$
Now choose $q$ such that $\int_{0}^{t} f^{\prime}(q e^{\beta(s)})\,ds=y-x$. Now denote $q=h(y-x, t),$ and then for every $y\in \mathbb{R}$ we have
\begin{equation}\label{equation 3.4}
e^{-\beta(t)} U (x, t)\leq U(y, 0))+\int_0 ^t e^{-\beta(s)}f^{*}(f^{\prime}(h(y-x, t) e^{\beta(s)}))\, ds.
\end{equation}

Now we prove that  equation \eqref{equation 3.4} is an equality for some $y\in\mathbb{R}.$
To do this, consider the curve $X(s)$ that satisfies
\begin{equation}\label{ode}
\begin{cases}
   X^{\prime}(s)=f^{\prime}(u(X(s), s))\\
   X(t)=x.
\end{cases}
\end{equation}
Then 
\begin{equation*}
    \frac{d}{ds}u(X(s), s)=u_t (X(s), s)+ u_x (X(s), s)f^{\prime}(u(X(s), s))=\alpha(s) u(X(s), s)).
\end{equation*}
 
This implies  $$U_x(X(s), s)=u(X(s), s)= u_0 (X(0))e^{\beta(s)},$$ 
where $\beta(s)=\int_0^s \alpha(u)du.$

  Now evaluating the equation $U_t +f(U_x)=\alpha(t) U$ at  $(X(s), s),$ and simplifying,  we get 
\begin{equation} 
\label{equation 3.5}
e^{-\beta(t)} U (x, t)= U(X(0), 0))+\int_0 ^t e^{-\beta(s)}f^{*}(f^{\prime}(u_0(X(0)) e^{\beta(s)}))ds.
\end{equation}
In fact, since $uf^{\prime}(u)-f^{*}(f^{\prime}(u))=f(u),$  $U_t (X(s), s)) +f(U_x(X(s), s))=\alpha(s) U(X(s), s)$ can be rewritten as 
$$U_t (X(s), s)) + U_x(X(s), s) f^{\prime}(u(X(s), s))- f^{*} (f^{\prime}(u(X(s), s)))=\alpha(s) U(X(s), s).$$
Multiplying $e^{-\beta(s)}$ on both sides of the above equation, we get
 $$\frac{d}{ds} \Big[e^{-\beta(s)}U (X(s), s)) \Big] =e^{-\beta(s)}f^{*} f^{\prime}((u(X(s), s))).$$

By integrating the above, we get \eqref{equation 3.5}. From equations \eqref{equation 3.4} and \eqref{equation 3.5}, we get 
\begin{equation*}
e^{-\beta(t)} U (x, t)=\displaystyle{ \min_{y\in \mathbb{R}}} \Big[U(y, 0))+\int_0 ^t e^{-\beta(s)}f^{*}(f^{\prime}(ye^{\beta(s)}))ds\Big].
\end{equation*}

Furthermore, integrating \eqref{ode} from $0$ to $t$, we get 
\begin{equation}
 \begin{aligned}
     x-X(0)&= \int_0 ^t f^{\prime}(u(X(s), s)) ds\\
           & =\int_0 ^t f^{\prime}(u_0 (X(0))e^{\beta(s)}) ds.
 \end{aligned}   
\end{equation}
By definition of $h$ function \eqref{h-function}, $u_0(X(0))=h(x-X(0), t).$ So the solution 
\begin{equation*}
\begin{aligned}
   u(x, t)&=u_0 (X(0))e^{\beta(t)}\\
         &= h(x-X(0), t)e^{\beta(t)}. 
\end{aligned}
\end{equation*}

Clearly $X(0)=y(x, t)$ is the unique minimizer of the functional $F(x, y, t)$  defined in the theorem.
To get the solution $\rho,$ observe that 
\begin{equation*}
    \frac{d}{ds}\rho(X(s), s)=\rho_t (X(s), s)+f^{\prime}(u(X(s)), s)\rho_x (X(s),s)=0
\end{equation*}
Therefore  $\rho(x,t)=\rho_0(X(0))=\rho_0(y(x,t)).$
\end{proof}
\section{Verification of the formula as a weak entropy solution}

Define the potential $F$ as follows:
\begin{equation}\label{potential}
F(x, t, y)=U(y, 0))+\int_0 ^t e^{-\beta(s)}f^{*}(f^{\prime}(ye^{\beta(s)}))ds.
\end{equation}
and the corresponding minimization
\begin{equation}\label{minimization}
F(x, t)=\displaystyle{ \min_{y\in \mathbb{R}}}F(x, t, y).
\end{equation}
In the following, we give a modified representation of the potential that is easy to use and suitable for our purpose. 
\begin{proposition}[\em {modified representation of potential}] \label{newrepresentation}
  The potential $F(x, y, t)$ given in \eqref{potential} is equivalent to 
  \begin{equation*}
 F(x,y,t)= F(x, 0, t)+ \int_{0}^y \big[u_0 (z)-h(x-z, t)\big] dz.
\end{equation*}
\end{proposition}
\begin{proof}
   From the property of Legendre transform, $uf^{\prime}(u)-f^{*}(f^{\prime}(u))=f(u).$ Then 
   \begin{equation*}
     f^{*}(f^{\prime}(h(x-y, t)e^{\beta(s)}))= e^{\beta(s)} h(x-y, t) f^{\prime}(h(x-y, t)e^{\beta(s)}) - f\Big(h(x-y, t)e^{\beta(s)}\Big). 
     \end{equation*}
   Therefore,
   \begin{equation}\label{4.3}
   \begin{aligned}
     F(x, y, t)=&\int_0^{y}u_0(z) dz+\int_{0} ^t f^*(f'(h(x-y, t)e^{\beta (s)}))e^{-\beta(s)}ds\\ 
               =&\int_0^ {y} u_0(z) dz + h(x-y, t)\int_0 ^t f^{\prime}(h(x-y, t)e^{\beta(s)}) ds\\
                  &- \int_0 ^t f\Big(h(x-y, t)e^{\beta(s)}\Big) e^{-\beta(s)} ds.
   \end{aligned}
     \end{equation}
 
     By the definition of h-function \eqref{h-function}, $F(x, y, t)$ in equation\eqref{4.3} becomes

       \begin{equation}\label{4.4}
         F(x, y, t)=\int_0^ {y} u_0(z) dz + (x-y) h(x-y, t)- \int_0 ^t f\Big(h(x-y, t)e^{\beta(s)}\Big) e^{-\beta(s)} ds. 
       \end{equation}

       Define $$G(x, y, t)=(x-y) h(x-y, t)- \int_0 ^t f\Big(h(x-y, t)e^{\beta(s)}\Big) e^{-\beta(s)} ds.$$
     Using the definition of h-function \eqref{h-function}, we obtain  $G_y(x, y, t)$ as follows.
    
    \begin{equation*}
    \begin{aligned}
         G_y (x, y, t)=& -h(x-y, t)- (x-y)h_y (x-y, t)\\
         &+ h_y (x-y, t)\int_0 ^t f^{\prime}\Big(h(x-y, t)e^{\beta(s)}\Big) ds\\
                      =& - h(x-y, t).
         \end{aligned}
       \end{equation*}
       
   Therefore, the potential $F$ can be re-written as

    \begin{equation*}
    \begin{aligned}
     F(x, y, t)&= \int_0^ {y} u_0(z) dz+ F(x, 0, t)+ \int_{0}^y  G_z (x, z, t) dz\\
               & =F(x, 0, t)+ \int_{0}^y  \big[u_0(z)-h(x-z, t)\big] dz.
    \end{aligned}
    \end{equation*}
    This completes the proof of the proposition.
\end{proof}

Due to the above proposition\eqref{potential}, the minimization\eqref{minimization} can be rephrased as 
\begin{equation} \label{rephrased-minimization}
F(x,t)= F(x, 0, t)+ \min_{y\in \mathbb{R}} \int_{0}^y  [u_0 (z)-h(x-z, t)] dz.
\end{equation}
The minimizer of the above\eqref{rephrased-minimization} exists, although in general it is not unique. This fact is the content of the theorem below, whose proof is a straightforward application of part $(3)$ of the proposition \eqref{properties of h-function}.
\begin{theorem}{(existence of minimizer)}
There exists a $y_0\in \mathbb{R}$ such that  $F(x,t)= F(x, 0, t)+  \int_{0}^{y_0}  [u_0 (z)-h(x-z, t)] dz.$
\end{theorem}
 Further, we prove that the minimizer is unique for almost every point $(x,t).$ To prove this fact, we need the following two Lemmas.
\begin{lemma}{(Non intersecting property)}\label{uniqueness at the interior}
    Let $y_0$ be a point where $F$ (\eqref{potential}) is minimum and $(\tilde{x}, \tilde{t})$ be any point on the curve 
    $X^{\prime}(s)=f^{\prime}(e^{\beta(s)}h(x-y_0)),\,\,\, X(0)=y_0,\,\,\ X(t)=x$, different from the end points $(x, t)$ and $(y,0)$. Then $F(\tilde{x}, \tilde{t}, y)$ has the unique minimum at $y_0$.
\end{lemma}
\begin{proof}
 Since $F(x,t, y)$ reaches the minimum at $y_0$, $F(x,t,y)-F(x,t, y_0)\geq 0$. Then  using proposition\eqref{newrepresentation}, we  deduce

 \begin{equation}\label{4.7}
     F(x, t, y)-F(x, t, y_0)=\int_{y_0}^{y} [u_0 (z)-h(x-z, t)] dz\geq 0
 \end{equation}
 Further 
  \begin{equation}
  \begin{aligned}
   & F(\tilde{x}, \tilde{t}, y)-F(\tilde{x}, \tilde{t}, y_0)\\
   &=\int_{y_0}^{y} [u_0 (z)-h(\tilde{x}-z, \tilde{t})] dz \\  
   &= \int_{y_0}^{y} [u_0 (z)-h(x-z, t)] dz + \int_{y_0}^{y} [h(x-z, t)-h(\tilde{x}-z, \tilde{t})] dz\\
   &=I_1 +I_2,
  \end{aligned}    
 \end{equation}
 where 
 \begin{equation*}
     \begin{aligned}
      I_1 &=  \int_{y_0}^{y} [u_0 (z)-h(x-z, t)] dz \\
      I_2 &= \int_{y_0}^{y} [h(x-z, t)-h(\tilde{x}-z, \tilde{t})] dz.
     \end{aligned}
 \end{equation*}
 Note that by equation\eqref{4.7},  $I_1\geq 0.$ Now we show that $I_2>0.$ To show this we consider the following two cases.\\

{\bf Case 1: $y< y_0$:} 
 Given $y< z< y_0$,  let $X_z$ be the $h$-curve joining $(x,t)$  and $(z, 0).$ Then by property $2$ of Theorem\eqref{properties of h-function},  $h(X_z (\tilde{t})-z, \tilde{t})=h(x-z, t).$ 
 In this case, 
 \begin{equation}
 \begin{aligned}
     I_2&= \int_{y_0}^{y} [h(x-z, t)-h(\tilde{x}-z, \tilde{t})] dz\\
     &= \int_{y_0}^{y} \Big[h(X_z (\tilde{t})-z, \tilde{t})-h(\tilde{x}-z, \tilde{t})\Big] dz\
 \end{aligned}
    \end{equation}
 Using $h(x, t)$ is a strictly increasing function in the variable $x$ for fixed $t$ (property $2$ of theorem\eqref{properties of h-function}), $I_2 >0.$\\
 \noindent
 {\bf Case 2} $y< y_0$:A Similar analysis can be applied to obtain $I_2 >0.$ This completes the proof.
\end{proof}
The minimizer for the potential $F$ in general is not unique. So we define the extreme points as follows:
\begin{align*}
    & y_*(x,t):=\min\{y\in \mathbb{R}: F(x,y,t)=\displaystyle{\min_{z\in \mathbb{R}}}F(x,z,t)\}\\
    & y^*(x,t):=\max\{y\in \mathbb{R}: F(x, y,t)=\displaystyle{\min_{z\in \mathbb{R}}}F(x,z,t)\}.
\end{align*}
\begin{proposition}\label{propeties of minimizer}
 The functions $y_{*}(x,t),\,\,y^{*}(x,t)$ satisfies the following.
 \begin{enumerate}
      \item If $x_1 <x_2,$ then $y_{*}(x_1,t) \leq y^{*}(x_1,t)\leq y_{*}(x_2,t).$
     \item $y_{*}(x,t),\,\,y^{*}(x,t)$ are monotone increasing functions in the variable $x.$
     \item For fixed $t,$ $y_{*}(x,t)=y^{*}(x,t)$ for all most every $x\in \mathbb{R}.$
     \item  $y_{*}(x,t)=y^{*}(x,t)$ for all most every $(x, t)\in \mathbb{R}\times (0, \infty).$
 \end{enumerate}
 \end{proposition}
 \begin{proof}
  Let $t$  be a fixed time and $x_1<x_2.$  Let $y_1$ be a point where $F(x_1, y, t)$ reaches its minimum value and $y_2$ be a point where $F(x_2, y, t)$ reaches its minimum value. If $y_2 <y_1,$ then the $h$-curve joining $(x_1, t)$ and $(y_1, 0)$ and the h-curve joining $(x_2, t)$ and $(y_2, 0)$ intersect each other at an interior point $(\bar{x}, \bar{t})$. Then the previous Lemma\eqref{uniqueness at the interior} implies $y_1=y_2,$ which is a contradiction. So $y_1\leq y_2.$ This implies the first and second statements of the proposition. 

  Since $y_{*}(x,t),\,\,y^{*}(x,t)$ are monotone increasing functions in the variable $x,$  except a set of measure zero, $y_{*}(x,t),\,\,y^{*}(x,t)$ are continuous.  By the first statement of the proposition, it follows that at the point of continuity of $y_{*}(x,t),\,\,y^{*}(x,t),$  we get  $y_{*}(x,t)=y^{*}(x,t).$  This proves the third statement. The fourth statement of the proposition follows directly from the third statement. This completes the proof of the proposition.
 \end{proof}
Now, we can prove the following theorem, which states that the formula given in \eqref{classicalsolution}.
\begin{theorem} The functional
  \begin{equation*} 
F(x, y, t)=\inf_{y\in \mathbb{R}}\Bigg\{\int_{0} ^t f^*(h(x-y, t)e^{\beta (\theta)})e^{-\beta(\theta)}d\theta+\int_0^{y}u_0(z) dz\Bigg\},
\end{equation*}
attains its minimum at a unique point $y(x,t)$ a.e. and $u(x, t)= h(x-y^{*}(x, t), t)e^{\beta(t)}$  satisfies the following weak formulation.

\begin{equation*}
    \int_{0}^{\infty} \int_{-\infty}^{\infty} [u \phi_t +f(u)\phi_x] \ dx \ dt +\int_{-\infty}^{\infty} u_0 (x)\phi(x,0) dx=-\int_{0}^{\infty} \int_{-\infty}^{\infty} \alpha(t) u(x,t)\phi(x,t)\ dx \ dt.
\end{equation*}
\end{theorem}

\begin{proof}
{\bf Step 1: ($F(x,t)$ is a Lipschitz continuous function in variable $x$)}
 Let $x_1 <x_2$. Then 
 \begin{equation*}
 \begin{aligned}
    &F(x_1, t)-F(x_2, t)\\
    &=F(x_1, y^{*}(x_1, t), t)-F(x_2, y^{*}(x_2, t), t)\\
     &= [F(x_1, y^{*}(x_1, t), t)-F(x_1, y^{*}(x_2, t), t)]+[F(x_1, y^{*}(x_2, t), t)- F(x_2, y^{*}(x_2, t), t)].
 \end{aligned}
 \end{equation*}
 The first term of the above expression is non-positive. Hence
 $$F(x_1, t)-F(x_2, t) < F(x_1, y^{*}(x_2, t), t)- F(x_2, y^{*}(x_1, t), t).$$

 Using the analysis as in the proof of Theorem\eqref{newrepresentation}, we get

 $$F(x, y, t)= F(0, y, t)-\int_{0}^{x} h(z-y,t) dz+ \int_{0}^y u_0(z) dz.$$

 Therefore
 \begin{equation}
 \label{weakderivativeinequality1}
 \begin{aligned}
    &F(x_1, t)-F(x_2, t)\\
    &\leq F(x_1, y^{*}(x_2, t), t)- F(x_2, y^{*}(x_2, t), t)\\
    & \leq F(0, y^{*}(x_2, t), t), t)+ \int_{0}^{x_1}h(z-y^{*}(x_2, t), t) dz + \int_{0}^{ y^{*}(x_2, t)} u_0(z) dz\\
    -& \Big[F(0, y^{*}(x_2, t), t), t)+ \int_{0}^{x_2}h(z-y^{*}(x_2, t), t) dz + \int_{0}^{ y^{*}(x_2, t)} u_0(z) dz\Big]
    =  \int_{x_2}^{x_1}h(z-y^{*}(x_2, t), t) dz.
 \end{aligned}
 \end{equation}
 Similarly,
\begin{equation}
\label{weakderivativeinequality2}
 \begin{aligned}
    &F(x_2, t)-F(x_1, t)\\
    &\leq F(x_2, y^{*}(x_2, t), t)- F(x_1, y^{*}(x_1, t), t)\\
    & \leq F(0, y^{*}(x_1, t), t), t)+ \int_{0}^{x_1}h(z-y^{*}(x_1, t), t) dz + \int_{0}^{ y^{*}(x_1, t)} u_0(z) dz\\
    -& \Big[F(0, y^{*}(x_1, t), t), t)+ \int_{0}^{x_1}h(z-y^{*}(x_1, t), t) dz + \int_{0}^{ y^{*}(x_1, t)} u_0(z) dz\Big]
    =  \int_{x_1}^{x_2}h(z-y^{*}(x_1, t), t) dz.
 \end{aligned}
 \end{equation}
 By \eqref{weakderivativeinequality1} and \eqref{weakderivativeinequality2}, we get
 
 \begin{equation}
 \label{lipfinalinx}
 \begin{aligned}
  \int_{x_2}^{x_1}h(z-y^{*}(x_2, t), t) dz\leq F(x_2, t)-F(x_1, t)\\\leq  \int_{x_1}^{x_2}h(z-y^{*}(x_1, t), t) dz.
 \end{aligned}
 \end{equation}
So $F(x, t)$ is a Lipschitz continuous function in the variable $x$.

\noindent
 {\bf Step 2: ($F(x,t)$ is a Lipschitz continuous function in the variable $t$}.)\\

Recall from \eqref{4.3} that 
 \begin{equation*}
         F(x, y, t)=\int_0^ {y} u_0(z) dz + (x-y) h(x-y, t)- \int_0 ^t f\Big(h(x-y, t)e^{\beta(s)}\Big) e^{-\beta(s)} ds. 
       \end{equation*}
To prove the Lipschitz continuity in the variable $t,$ we rewrite the potential in a convenient form: To do this, define 
$$G(x,y, t)=(x-y) h(x-y, t)- \int_0 ^t f\Big(h(x-y, t)e^{\beta(s)}\Big) e^{-\beta(s)} ds.$$
 Then the potential $F$ can be re-written as
 $$F(x,y,t)=\int_0^yu_0(z)\,dz+G(x,y,t).$$
 Now
 \begin{equation}
 \label{calculationforpartialderivativewrtt}
 \begin{aligned}
 \frac{\partial G}{\partial t}=&(x-y) h_t(x-y, t)-f\Big(h(x-y, t)e^{\beta(t)}\Big) e^{-\beta(t)}\\
                              &-\int_{0}^t f^{\prime}\Big(h(x-y, t)e^{\beta(s)}\Big)h_t(x-y, t)ds\\
                              =& (x-y) h_t(x-y, t)-f\Big(h(x-y, t)e^{\beta(t)}\Big) e^{-\beta(t)}-(x-y) h_t(x-y, t)\\
                              =&-f\Big(h(x-y, t)e^{\beta(t)}\Big) e^{-\beta(t)}.
\end{aligned}
 \end{equation}
 
 So for any $t>t_0>0$ 
 \begin{equation}
 \label{expessionusefulint}
 \begin{aligned}
 G(x, y,t)&=G(x, y,t_0)+ \int_{t_0}^{t} \frac{\partial G}{\partial t}(x, y, \tau)d\tau\\
          &=G(x, y,t_0)-\int_{t_0}^{t}f\Big(h(x-y, \tau)e^{\beta(\tau)}\Big) e^{-\beta(\tau)}d\tau
 \end{aligned}
 \end{equation}
 Let $t_0<t_1 <t_2$. Then 
 \begin{equation*}
 \begin{aligned}
    &F(x, t_1)-F(x, t_2)\\
    &=F(x, y^{*}(x, t_1), t_1)-F(x, y^{*}(x, t_2), t_2)\\
     &= [F(x, y^{*}(x, t_1), t_1)-F(x, y^{*}(x, t_2), t_2)]+[F(x, y^{*}(x, t_1), t_2)- F(x, y^{*}(x, t_2), t_2)].
 \end{aligned}
 \end{equation*}
 The second term of the above expression is non-negative. Hence using \eqref{expessionusefulint}
 \begin{equation}
 \label{lipschitz-1}
 \begin{aligned}
 F(x, t_1)-F(x, t_2) \geq& F(x, y^{*}(x, t_1), t_1)-F(x, y^{*}(x, t_1), t_2)\\
                      =&G(x,  y^{*}(x, t_1),t_0)-\int_{t_0}^{t_1}f\Big(h(x- y^{*}(x, t_1), \tau)
                     e^{\beta(\tau)}\Big) e^{-\beta(\tau)}d\tau\\
                     &-\Big[G(x,  y^{*}(x, t_1),t_0)-\int_{t_0}^{t_2}f\Big(h(x- y^{*}(x, t_1), \tau)
                     e^{\beta(\tau)}\Big) e^{-\beta(\tau)}d\tau  \Big]\\
                     =&\int_{t_1}^{t_2}f\Big(h(x- y^{*}(x, t_1), \tau)
                     e^{\beta(\tau)}\Big) e^{-\beta(\tau)}d\tau.
 \end{aligned}
 \end{equation}

 Similarly,

 \begin{equation*}
 \begin{aligned}
    &F(x, t_2)-F(x, t_1)\\
    &=F(x, y^{*}(x, t_2), t_2)-F(x, y^{*}(x, t_1), t_1)\\
     &= [F(x, y^{*}(x, t_2), t_2)-F(x, y^{*}(x, t_2), t_1)]+[F(x, y^{*}(x, t_2), t_1)- F(x, y^{*}(x, t_1), t_1)].
 \end{aligned}
 \end{equation*}
 The second term of the above expression is non-negative. Hence using \eqref{expessionusefulint}
 \begin{equation}
 \label{lipschitz-2}
 \begin{aligned}
 F(x, t_2)-F(x, t_1) \geq& F(x, y^{*}(x, t_1), t_2)- F(x, y^{*}(x, t_1), t_1)\\
                      =&G(x,  y^{*}(x, t_2),t_0)-\int_{t_0}^{t_2}f\Big(h(x- y^{*}(x, t_2), \tau)
                     e^{\beta(\tau)}\Big) e^{-\beta(\tau)}d\tau\\
                     &-\Big[G(x,  y^{*}(x, t_2),t_0)-\int_{t_0}^{t_1}f\Big(h(x- y^{*}(x, t_2), \tau)
                     e^{\beta(\tau)}\Big) e^{-\beta(\tau)}d\tau  \Big]\\
                     =&\int_{t_2}^{t_1}f\Big(h(x- y^{*}(x, t_2), \tau)
                     e^{\beta(\tau)}\Big) e^{-\beta(\tau)}d\tau.
 \end{aligned}
 \end{equation}

 By \eqref{lipfinalinx}, \eqref{lipschitz-1} and \eqref{lipschitz-2}, it follows that $F$ is Lipschitz continuous in variable $t$.

{\bf Step 3:(Conclusion)} If $(x,t)$ is a point of continuity for $F, y^{*}$ and $y_{*}$, then one gets from \eqref{lipschitz-1} and \eqref{lipschitz-2}
 
  \begin{equation}
\label{Fx and Ft}
    F_x (x,t)= h(x-y(x,t), t)\,\,\textnormal{and}\,\,  F_t(x,t)=-f\Big(h(x- y(x, t),t)e^{\beta(t)}\Big)e^{-\beta(t)}.
\end{equation}

 Let us define $u:\mathbb{R}\times[0,\infty)\to \mathbb{R}$ as
\begin{equation*}
    u(x,t):=h(x-y(x,t),t)e^{\beta(t)}.
\end{equation*}
Let us take any test function $\phi\in C_c^\infty (\mathbb{R}\times [0,\infty))$. Then by \eqref{Fx and Ft}, we get
\begin{equation*}
    \begin{aligned}
        &\int_{0}^{\infty} \int_{-\infty}^{\infty} [u \phi_t +f(u)\phi_x] \ dx \ dt \\
        =&\int_{0}^{\infty} \int_{-\infty}^{\infty}F_x(x,t)e^{\beta(t)}\phi_tdxdt-\int_{0}^{\infty} \int_{-\infty}^{\infty}F_t(x,t)\phi_xe^{\beta(t)}dxdt\\
        =&-\int_{0}^{\infty} \int_{-\infty}^{\infty}F(x,t)e^{\beta(t)}\phi_{xt}(x,t)dxdt+\int_{0}^{\infty} \int_{-\infty}^{\infty}F(x,t)\alpha(t)e^{\beta(t)}\phi_x dxdt\\&+\int_{0}^{\infty} \int_{-\infty}^{\infty}F(x,t)e^{\beta(t)}\phi_{xt} dxdt+\int_{-\infty}^\infty \phi_x(x,0)F(x,0)dx\\
        =&\int_{0}^{\infty} \int_{-\infty}^{\infty}F(x,t)\alpha(t)e^{\beta(t)}\phi_x dxdt+\int_{-\infty}^\infty \phi_x(x,0)F(x,0)dx\\
        =& -\int_{0}^{\infty} \int_{-\infty}^{\infty}F_x(x,t)\alpha(t)e^{\beta(t)}\phi dxdt-\int_{-\infty}^\infty \phi(x,0)F_x(x,0)dx\\
        =& -\int_{0}^{\infty} \int_{-\infty}^{\infty}h(x-y(x,t),t)\alpha(t)e^{\beta(t)}\phi dxdt-\int_{-\infty}^\infty \phi(x,0)h(x-y(x,0),0)dx\\
        =& -\int_{0}^{\infty} \int_{-\infty}^{\infty}h(x-y(x,t),t)e^{\beta(t)}\alpha(t)\phi dxdt-\int_{-\infty}^\infty \phi(x,0)u_0(x)dx\\
        =& -\int_{0}^{\infty} \int_{-\infty}^{\infty}\alpha(t)u(x,t)\phi(x,t)dxdt-\int_{-\infty}^\infty u_0(x)\phi(x,0)dx.
    \end{aligned}
\end{equation*}
Therefore,
\begin{equation*}
 \int_{0}^{\infty} \int_{-\infty}^{\infty} [u \phi_t +f(u)\phi_x] \ dx \ dt +\int_{-\infty}^{\infty} u_0 (x)\phi(x,0) dx=-\int_{0}^{\infty} \int_{-\infty}^{\infty} \alpha(t) u(x,t)\phi(x,t)\ dx \ dt.   
\end{equation*}
\end{proof}

\section{Uniqueness of solution of the System}
The uniqueness of the solution of the first equation of   \eqref{equation for R} is known, see \cite{manish}.
To analyze the second equation of the system \eqref{equation for R}, we first introduce the notion of characteristic triangles, which will serve as the foundation for constructing forward generalized characteristics\cite{MR457947, dcm,mcd, Wang01, WHD97}.
\begin{definition}
   Given a point $(x,t)$, where $x \in \mathbb{R}$ and $t>0$, we define the characteristic triangle $\triangle (x,t)$ as the region above the $X$-axis enclosed between two $h$-curves: one connecting $(x,t)$ to $(y_{*}(x,t),0)$, and the other connecting $(x,t)$ to $(y^{*}(x,t),0)$.
\end{definition}
The following lemma is an easy consequence of the non-intersecting property of characteristics(Lemma\eqref{uniqueness at the interior}).
\begin{lemma}
\label{nonintersecting property of ch triangles}
    For fixed $t>0$, $\bigcup_{x\in \mathbb{R}}\triangle(x,t)=\mathbb{R}\times (0, t]$ and two distinct characteristic triangles never intersect each other.
\end{lemma}
\begin{theorem}
Given any point $(x_0, t_0),$  where $x_0\in\mathbb{R},\,\, t_0>0,$ there exists a Lipschitz continuous curve $X(t)$ such that 
$X(t_0)=x_0$ and 
\begin{equation}
    \lim_{t^{\prime \prime}, t^{\prime}\to t+}\frac{X(t^{\prime \prime})-X(t^{\prime})}{t^{\prime \prime}-t^{\prime}}=
    \begin{cases}
    f'\!\left(h(x-y(x,t),t)e^{\beta(t)}\right),
    & \text{if } y_*(x,t)=y^*(x,t),\\[1ex]
    \displaystyle
    \frac{f(u(x+,t))-f(u(x-,t))}
         {u(x+,t)-u(x-,t)},
    & \text{if } y_*(x,t)\neq y^*(x,t).
    \end{cases}
\end{equation}

    
\end{theorem}
\begin{proof}
    We define the curve $X$ as follows. Fix $t>t_0$. By Lemma~\eqref{nonintersecting property of ch triangles}, there exists a unique $x\in \mathbb{R}$ such that $(x_0,t_0)\in \triangle(x,t)$. Equivalently, there exists a unique $x\in \mathbb{R}$ satisfying
\[
y_{*}(x,t)\leq x_0\leq y^{*}(x,t).
\]
We then define
\[
X(t;x_0):=x.
\]

    \noindent
    \textbf{Case 1} (Degenerate Case  $y^*(x,t)=y_*(x,t)$):\\

    \noindent
Suppose $t^{\prime \prime}>t^{\prime}\geq t,$ let $X(t^{\prime})=x^{\prime}$ and $X(t^{\prime \prime})=x^{\prime \prime}$.  Let $X_1$ be the h-curve joining $(x^{\prime \prime},t^{\prime \prime})$ and $(y_*(x^{\prime \prime},t^{\prime \prime}),0)$ and $X_2$ be the h-curve joining $(x^{\prime},t^{\prime})$ and $(y^*(x^{\prime},t^{\prime}),0)$. Since $X_1(t)\leq X(t)\leq X_2(t),$ 
    \begin{equation}
    \label{slope inequality}
     \frac{X_1(t^{\prime \prime})-X_1(t^{\prime})}{t^{\prime \prime}-t^{\prime}} \geq \frac{X(t^{\prime \prime})-X(t^{\prime})}{t^{\prime \prime}-t^{\prime}} \geq  \frac{X_2(t^{\prime \prime})-X_2(t^{\prime})}{t^{\prime \prime}-t^{\prime}}.  
    \end{equation}

  Now, from the definition of $h$-function
  \begin{equation}
  \label{Integral form h function}
      \begin{aligned}
        \frac{X_1(t^{\prime \prime})-X_1(t^{\prime})}{t^{\prime \prime}-t^{\prime}}&=
        \frac{1}{t^{\prime \prime}-t^{\prime}}\int_{t^{\prime}}^{t^{\prime \prime}} f^{\prime}(h(x-y_*(x^{\prime \prime}, t^{\prime \prime})e^{\beta(s)}) ds\\
         \frac{X_2(t^{\prime \prime})-X_2(t^{\prime})}{t^{\prime \prime}-t^{\prime}}&=
        \frac{1}{t^{\prime \prime}-t^{\prime}}\int_{t^{\prime}}^{t^{\prime \prime}} f^{\prime}(h(x-y_*(x^{\prime}, t^{\prime})e^{\beta(s)}) ds.\\
      \end{aligned}
  \end{equation}
  Using the relation\eqref{Integral form h function} in \eqref{slope inequality} and passing to the limit as $t^{\prime \prime}, t^{\prime}\to t,$ we get  
    \begin{equation*}
     \lim_{t^{\prime \prime}, t^{\prime}\to t+}\frac{X(t^{\prime \prime})-X(t^{\prime})}{t^{\prime \prime}-t^{\prime}}= f'(h(x-y(x,t),t)e^{\beta(t)}).   
 \end{equation*}\\

 \noindent
\textbf{Case 2}(Non-degenerate Case $y^*(x,t) \neq y_*(x,t)$):\\  

 Recall the potential functional $F(x,y,t)$ 
     \begin{equation*}
     \begin{aligned}
      F(x,y,t)=(x-y)h(x-y,t)- \int_0^t f\big(h(x-y,t)e^{\beta (\theta)}\big)e^{-\beta (\theta)}d\theta + \int _0^y u_0(z)dz.\\
      \end{aligned}  
     \end{equation*}
     From the earlier calculations, $$F_t(x,y,t)=-f(h(x-y,t)e^{\beta (t)})e^{-\beta (t)}, ~~ F_x(x,y,t)=h(x-y,t).$$
     Suppose $t^{\prime \prime}>t^{\prime}\geq t,$ let $X(t^{\prime})=x^{\prime}$ and $X(t^{\prime \prime})=x^{\prime \prime}$. Observe that
     \begin{equation*}
         \begin{aligned}
             &F(x^{\prime},y_{*}(x^{\prime},t^{\prime}),t^{\prime})-F(x^{\prime},y^{*}(x^{\prime \prime},t^{\prime \prime}),t^{\prime}) \leq 0\\
             \text{and }&F(x^{\prime \prime},y_{*}(x^{\prime},t^{\prime}),t^{\prime \prime})-F(x^{\prime \prime},y^{*}(x^{\prime \prime},t^{\prime \prime}),t^{\prime \prime}) \geq 0.
         \end{aligned}
     \end{equation*}
     Therefore,
     $$F(x',y_{*}(x^{\prime},t^{\prime}),t')-F(x',y^*(x'',t''),t') \leq F(x'',y_{*}(x^{\prime},t^{\prime}),t'')-F(x'',y^*(x'',t''),t'').$$
     This implies
     \begin{equation}
     \label{inequality-potential}
      F(x^{\prime},y_{*}(x^{\prime},t^{\prime}),t^{\prime})-F(x^{\prime \prime},y_{*}(x^{\prime},t^{\prime}),t^{\prime \prime}) \leq F(x^{\prime},y^{*}(x^{\prime \prime},t^{\prime \prime}),t^{\prime})-F(x^{\prime \prime},y^{*}(x^{\prime \prime},t^{\prime \prime}),t^{\prime \prime}).   
     \end{equation}
     Now the left-hand side of \eqref{inequality-potential} can be written as,
     \begin{equation*}
         \begin{aligned}
          &F(x^{\prime},y_{*}(x^{\prime},t^{\prime}),t^{\prime})-F(x^{\prime \prime},y_{*}(x^{\prime},t^{\prime}),t^{\prime \prime})\\
      =&[F(x^{\prime},y_{*}(x^{\prime},t^{\prime}),t^{\prime})-F(x^{\prime},y_{*}(x^{\prime},t^{\prime}),t^{\prime \prime})]+[F(x^{\prime},y_{*}(x^{\prime},t^{\prime}),t^{\prime \prime})-F(x^{\prime \prime},y_{*}(x^{\prime},t^{\prime}),t^{\prime \prime})]\\
      =&\int_{t^{''}}^{t'}F_t(x',y_{*}(x^{\prime},t^{\prime}),t)dt + \int_{x^{''}}^{x'}F_x(z,y_{*}(x^{\prime},t^{\prime}),t'')dz\\
      =&\int_{t'}^{t^{\prime \prime}}f(h(x-y_{*}(x^{\prime},t^{\prime}),t)e^{\beta(t)})e^{-\beta(t)}dt + \int_{x''}^{x'}h(z-y_{*}(x^{\prime},t^{\prime}),t'')dz.   
         \end{aligned}
     \end{equation*}
         
   Dividing the above equation by $(t^{\prime \prime}-t^{\prime}),$ we get
   \begin{equation}
   \label{nondegenerate-1}
   \begin{aligned}
   &\frac {F(x',y_{*}(x^{\prime},t^{\prime}),t')-F(x'',y_{*}(x^{\prime},t^{\prime}),t'')}{t^{\prime \prime}-t^{\prime}}\\
     &=\frac{1}{t^{\prime \prime}-t^{\prime}} \int_{t^{\prime}}^{t^{\prime \prime}}f(h(x-y_{*}(x^{\prime},t^{\prime}),t)e^{\beta(t)})e^{-\beta(t)}dt - \frac{x^{\prime \prime}-x^{\prime}}{t^{\prime \prime}-t^{\prime}} \frac{1}{x^{\prime \prime}-x^{\prime}} \int_{x^{\prime}}^{x^{\prime \prime}}h(z-y_{*}(x^{\prime},t^{\prime}),t^{\prime \prime})dz. 
     \end{aligned}
   \end{equation}

   Similarly, the right-hand side of \eqref{inequality-potential} can be modified

   \begin{equation}
   \label{nondegenerate-2}
   \begin{aligned}
   &\frac{F(x',y^{*}(x'',t''),t')-F(x'',y^{*}(x'',t''),t'')}{t^{\prime \prime}-t^{\prime}}\\
     &=\frac{1}{t^{\prime \prime}-t^{\prime}} \int_{t^{\prime}}^{t^{\prime \prime}}f(h(x-y^{*}(x^{\prime \prime},t^{\prime \prime}),t)e^{\beta(t)})e^{-\beta(t)}dt - \frac{x^{\prime \prime}-x^{\prime}}{t^{\prime \prime}-t^{\prime}} \frac{1}{x^{\prime \prime}-x^{\prime}} \int_{x^{\prime}}^{x^{\prime \prime}}h(z-y^{*}(x^{\prime \prime},t^{\prime \prime}),t^{\prime \prime})dz. 
     \end{aligned}
   \end{equation}
From \eqref{inequality-potential}, \eqref{nondegenerate-1} and \eqref{nondegenerate-2}, we get
\begin{equation}
\begin{aligned}
 & \frac{1}{t^{\prime \prime}-t^{\prime}}\Big[  \int_{t^{\prime}}^{t^{\prime \prime}}f(h(x-y_{*}(x^{\prime},t^{\prime}),t)e^{\beta(t)})e^{-\beta(t)}dt -\int_{t^{\prime}}^{t^{\prime \prime}}f(h(x-y^{*}(x^{\prime \prime},t^{\prime \prime}),t)e^{\beta(t)})e^{-\beta(t)}dt\Big] \\
 &\leq  \frac{x^{\prime \prime}-x^{\prime}}{t^{\prime \prime}-t^{\prime}} \frac{1}{x^{\prime \prime}-x^{\prime}}\Big[ \int_{x^{\prime}}^{x^{\prime \prime}}h(z-y_{*}(x^{\prime},t^{\prime}),t^{\prime \prime})dz-\int_{x^{\prime}}^{x^{\prime \prime}}h(z-y^{*}(x^{\prime \prime},t^{\prime \prime}),t^{\prime \prime})dz\Big].
\end{aligned}    
\end{equation}

Now passing to the limit as $t^{\prime \prime},t^{\prime} \to t,$  in the above equation, we get
 \begin{equation*}
 \begin{aligned}
    &(f(h(x-y_{*}(x,t),t)e^{\beta(t)})-f(h(x-y^{*}(x,t),t)e^{\beta(t)})e^{-\beta(t)} \\
    &\leq \Big[ h(x-y_{*}(x,t),t)-h(x-y^{*}(x,t),t)\Big] \liminf_{t^{\prime}, t^{\prime \prime} \to t} \frac{x^{\prime \prime}-x^{\prime}}{t^{\prime \prime}-t^{\prime}}.
    \end{aligned}
 \end{equation*}
Following a similar analysis, we get

\begin{equation*}
 \begin{aligned}
    &(f(h(x-y_{*}(x,t),t)e^{\beta(t)})-f(h(x-y^{*}(x,t),t)e^{\beta(t)})e^{-\beta(t)} \\
    &\geq \Big[ h(x-y_{*}(x,t),t)-h(x-y^{*}(x,t),t)\Big] \limsup_{t^{\prime}, t^{\prime \prime} \to t} \frac{x^{\prime \prime}-x^{\prime}}{t^{\prime \prime}-t^{\prime}}.
    \end{aligned}
 \end{equation*}

From the above two equations, the proof for the nondegenerate case follows.

    \end{proof}
    Based on the above property of generalized characteristic, we redefine $u$ as follows.
\begin{equation}
\label{newdefinition of u}
    u(x,t)= \begin{cases}
    \!h(x-y(x,t),t)e^{\beta(t)},
    & \text{if } y_*(x,t)=y^*(x,t),\\[1ex]
   (f^{\prime})^{-1} \Big(\displaystyle
    \frac{f(u(x+,t))-f(u(x-,t))}
         {u(x+,t)-u(x-,t)}\Big),
    & \text{if } y_*(x,t)\neq y^*(x,t).
    \end{cases} 
\end{equation}
    
\begin{theorem} With the new definition of $u$ $\eqref{newdefinition of u}$, the function  $R(x,t)=\int_0^{y_*(x,t)}\rho_0(y)dy$ satisfies the second equation of the system\eqref{equation for R} in the sense of the definition \eqref{weak formulation of R}.
\end{theorem}
\begin{proof} {\bf Step 1:}First assume that $\rho_0 >0.$ Define the functionals $R, E$  as follows:

\begin{equation*}
    \begin{aligned}
        R(x,t)&=\int_0^{y_*(x,t)}\rho_0(y)dy\\
        E(x,z,t)&=\int_0^z \rho_0(y)(X(y,t)-x)dy,\\
        E(x,t)&=\int_0^{y_*(x,t)}\rho_0(y)(X(y,t)-x)dy.
    \end{aligned}
\end{equation*}
Now our claim is:
\begin{equation*}
     E(x,t)=\min_{z\in \mathbb{R}}E(x,z,t).
\end{equation*}
    In fact, for $z<y_*(x,t)$, we have $X(z,t)<x$.
    Therefore,
    \begin{equation*}
        E(z,t)-E(y_*(x,t),t)=\int_{y_*(x,t)}^z\rho_0(y)(X(y,t)-x)dy\geq 0.
    \end{equation*}
    For $y_*(x,t)\leq z\leq y^*(x,t)$, $X(z,t)=x$. Thus,
    \begin{equation*}
      E(z,t)-E(y_*(x,t),t)=\int_{y_*(x,t)}^z\rho_0(y)(X(y,t)-x)dy=0.   
    \end{equation*}
    Again, for $z>y^*(x,t)$, we have $X(z,t)>x$. Therefore,
\begin{equation*}
     E(z,t)-E(y_*(x,t),t)=\int_{y_*(x,t)}^z\rho_0(y)(X(y,t)-x)dy\geq 0.
\end{equation*}
 So, $E(x,z,t)$ achieves a minimum at $z=y_*(x,t)$. \\  

 \noindent
{\bf Step 2:} In this step, we compute $E_x$ and $E_t$ using the minimization formulation. Let $x_1,x_2\in \mathbb{R}$ and $x_1<x_2$. By step $1,$  
\begin{equation*}
    E(x_1, t)=\int_0^{y_*(x_1,t)}\rho_0(y)(X(y,t)-x_1)dy\leq \int_0^{y_*(x_2,t)}\rho_0(y)(X(y,t)-x_1)dy.
    \end{equation*}
Therefore,
\begin{equation*}
    \begin{aligned}
        &E(x_2,t)-E(x_1,t)\\
       &\geq \int_0^{y_*(x_2,t)}\rho_0(y)(X(y,t)-x_2)dy-\int_0^{y_*(x_2,t)}\rho_0(y)(X(y,t)-x_1)dy\\
        =&-\int_0^{y_*(x_2,t)}\rho_0(y)(x_2-x_1)dy\\
        =&-(x_2-x_1)R(x_2,t).
    \end{aligned}
\end{equation*}

 Applying $\liminf_{x_1, x_2 \to x},$ we get
 $$ \liminf _{x_2, x_1 \to x} \frac{E(x_2, t)-E(x_1,t)}{x_2-x_1}\geq -\int_{0} ^{y^{*}(x,t)} \rho_0 (y) dy.$$
    
Similarly, one can show the following.

$$ \limsup _{x_2, x_1 \to x} \frac{E(x_2, t)-E(x_1,t)}{x_2-x_1}\leq -\int_{0} ^{y_{*}(x,t)} \rho_0 (y) dy.$$

Since \(y_{*}(x,t)=y^{*}(x,t)\) a.e.,  \(E\) is differentiable a.e. and it follows from the above two equations that
\begin{equation}
\label{value of R}
    E_{x}(x,t)=-R(x,t) \quad \text{a.e.}
\end{equation}

Next, we compute $E_t$. Let $t_1,t_2\in[0,\infty)$ and $t_1<t_2$. Then
\begin{equation*}
    \begin{aligned}
        &E(x,t_1)-E(x,t_2)\\
        =&\int_0^{y_*(x,t_1)}\rho_0(y)(X(y,t_1)-x)dy-\int_0^{y_*(x,t_2)}\rho_0(y)(X(y,t_2)-x)dy\\
        \leq &\int_0^{y_*(x,t_2)}\rho_0(y)(X(y,t_1)-x)dy-\int_0^{y_*(x,t_2)}\rho_0(y)(X(y,t_2)-x)dy\\
        =&\int_0^{y_*(x,t_2)}\rho_0(y)(X(y,t_1)-X(y,t_2))dy.
    \end{aligned}
\end{equation*}
So, dividing both sides by $(t_1-t_2)$ we get
\begin{equation*}
    \frac{E(x,t_1)-E(x,t_2)}{t_1-t_2}\leq \int_0^{y_*(x,t_2)}\rho_0(y)\frac{X(y,t_1)-X(y,t_2)}{t_1-t_2}dy.
\end{equation*}
If $(x,t)$ is a point of continuity for $y_*$, $y^{*}$ and point of differentiability for $E$, then applying $t_1, t_2 \to t$ to both sides of the above equation, we obtain
\begin{equation*}
\begin{aligned}
  E_t(x,t)&\leq \int_0^{y(x,t)}\rho_0(y)\dot X(y,t)dy\\
  &=\int_0^{y(x,t)}\rho_0(y)f'(u(X(y,t),t))dy.
\end{aligned}  
\end{equation*}
Similarly, one can show $E_t(x,t)\geq \int_0^{y(x,t)}\rho_0(y)f'(u(X(y,t),t))dy$. So we finally get
\begin{equation}
\label{Value of Et}
  E_t(x,t)= \int_0^{y_{*}(x,t)}\rho_0(y)f'(u(X(y,t),t))dy \quad \text{a.e.} 
\end{equation}

{\bf Step 3:} In this step, we want to express the measure $dE_t$ in terms of the measure $dR$ using the Radon-Nikodym theorem.
To show this, we introduce $M$ as follows: $M:\mathbb{R}\times[0,\infty)\to \mathbb{R}$ as 
$$M(x,t)=\int_0^{y(x,t)}\rho_0(y)\big[f'(u(X(y,t),t))+N\big]dy,$$ 
where $N$ is a real number such that $|f'(u(x,t))|\leq N.$ 
Now $M$ is a monotonically increasing function. One can show that the Lebesgue-Stieltjes measures $dM$ and $dR$ satisfy  $dM \ll dR.$ By the Radon-Nikodym theorem, there exists a Borel measurable function $g$ such that $\frac{dM}{dR}=g.$  We now find the pointwise values of $g.$
 Let $x\in \mathbb{R}$. For $x_1<x< x_2$ we have 
\begin{equation}
\label{RN derivative}
\begin{aligned}
 \dfrac{dM((x_1, x_2])}{dR((x_1, x_2])}&=\frac{M(x_2,t)-M(x_1,t)}{R(x_2,t)-R(x_1,t)}\\
    &=\dfrac {\int_{y_*(x_1,t)}^{y_*(x_2,t)}\rho_0(y)f'(u(X(y,t),t)dy}{\int_{y_*(x_1,t)}^{y_*(x_2,t)}\rho_0(y)dy}+ N
\end{aligned}
\end{equation}
Suppose that at $(x,t)$, the minimizer is unique. Then as $x_1,x_2\to x$, we get from the above equation\eqref{RN derivative} by employing the Lebesgue differentiation theorem, 

$$\lim_{x_1, x_2 \to x}\dfrac{dM((x_1, x_2])}{dR((x_1, x_2])}=f'(u(x,t))+N.$$

Suppose that at $(x,t)$, the minimizer is not unique. Then as $x_1,x_2\to x$, we get from the above equation\eqref{RN derivative}

\begin{equation*}
\begin{aligned}
 \lim_{x_1, x_2 \to x}\dfrac{dM((x_1, x_2])}{dR((x_1, x_2])}&=\frac{M(x_2,t)-M(x_1,t)}{R(x_2,t)-R(x_1,t)}\\
&=\dfrac {\int_{y_*(x,t)}^{y^*(x,t)}\rho_0(y)f'(u(X(y,t),t)dy}{\int_{y_*(x,t)}^{y^*(x,t)}\rho_0(y)dy}+ N.\\
\end{aligned}
\end{equation*}
As $X(y,t)=x$ for $y_*(x,t) \leq y \leq y^*(x,t),$ the above implies 
$$\lim_{x_1, x_2 \to x}\dfrac{dM((x_1, x_2])}{dR((x_1, x_2])}=f'(u(x,t))+N.$$
So we have $\dfrac{dM}{dR}=f'(u)+N.$ Thus, for any test function $\phi \in C_{c} ^{\infty} \Big(\mathbb{R}\times (0,\infty)\Big)$ we can write
\begin{equation*}
   \int_0^{\infty} \int_{-\infty}^{\infty}\phi(x,t)dM(x)dt=\int_0^{\infty}\int_{-\infty}^{\infty}\phi(x,t)(f'(u(x,t))+N)dR(x)dt.
\end{equation*}
    That is,
    \begin{align*}
    -\int_0^{\infty}\int_{-\infty}^{\infty}\phi_x(x,t)M(x,t)dxdt=\int_0^{\infty}\int_{-\infty}^{\infty}\phi(x,t)f'(u(x,t))dR(x)dt \\+ N\int_0^{\infty}\int_{-\infty}^{\infty}\phi(x,t)dR(x)dt.
     \end{align*}
    So, putting the value of M, we obtain the following.
    \begin{equation*}
    \begin{aligned}
    &-\int_0^{\infty}\int_{-\infty}^{\infty}\phi_x(x,t)\Big[\int_0^{y(x,t)}\rho_0(y)\big(f'(u(X(y,t),t))+N\big)dy\Big]dxdt\\&=\int_0^{\infty}\int_{-\infty}^{\infty}\phi(x,t)f'(u(x,t))dR(x)dt +N\int_0^{\infty}\int_{-\infty}^{\infty}\phi(x,t)dR(x)dt.    
    \end{aligned} 
    \end{equation*}
    As the term with N cancels out, we finally get
    \begin{equation*}
     -\int_0^{\infty}\int_{-\infty}^{\infty}\phi_x(x,t)\int_0^{y(x,t)}\rho_0(y)f'(u(X(y,t),t))dydxdt=\int_0^{\infty}\int_{-\infty}^{\infty}\phi(x,t)f'(u(x,t))dR(x)dt.   
    \end{equation*}
    Using \eqref{Value of Et} we can write
    \begin{equation*}
     -\int_0^{\infty}\int_{-\infty}^{\infty}\phi_x(x,t)E_t(x,t)dxdt=\int_0^{\infty}\int_{-\infty}^{\infty}\phi(x,t)f'(u(x,t))dR(x)dt .  
    \end{equation*}
    Using Fubini's theorem and integrating by parts,  we get
    \begin{equation*}
    \int_0^{\infty}\int_{-\infty}^{\infty}\phi_{xt}(x,t)E(x,t)dxdt =\int_0^{\infty}\int_{-\infty}^{\infty}\phi(x,t)f'(u(x,t))dR(x)dt .    
    \end{equation*}
    Again, integrating by parts with respect to the $x$ variable, we get
    \begin{equation*}
     -\int_0^{\infty}\int_{-\infty}^{\infty}\phi_t(x,t)E_x(x,t)dxdt=\int_0^{\infty}\int_{-\infty}^{\infty}\phi(x,t)f'(u(x,t))dR(x)dt.    
    \end{equation*}
    Now, using \eqref{value of R}, we conclude
    \begin{equation}
        \int_0^{\infty}\int_{-\infty}^{\infty}R(x,t)\phi_t(x,t)dxdt-\int_0^{\infty}\int_{-\infty}^{\infty}\phi(x,t)f'(u(x,t))dR(x)dt=0.
    \end{equation}
  So $R$ is a weak solution to the above PDE.\\
  {\bf Step 4:} Now, in general, let us take $\rho_0$ to be a bounded measurable function on $\mathbb{R}$. Then there exist two non-negative functions $\rho_0^+$ and $\rho_0^-$ in $\mathbb{R}$ such that
  $$\rho_0(y)=(\rho_0^+(y)+\epsilon)-(\rho_0^-(y)+\epsilon) \, \forall y\in \mathbb{R}.$$
  Note that $\rho_0^++\epsilon$ and $\rho_0^-+\epsilon$ are strictly positive functions. Therefore, following the above process corresponding to $\rho_0^++\epsilon$ and $\rho_0^-+\epsilon$  we can find weak solutions $R_1$ and $R_2$ respectively. As a result,  we get the weak solution $R=R_1-R_2$ for a general bounded measurable function $\rho_0$.\\
\noindent
  {\bf Step 5:}
     \begin{equation*}
  \begin{aligned}
     |R(x,t)-R_0(x)|&=|R_0 (y(x,t))-R_0 (x)|\\
                               &\leq L |y(x,t)-x|\\
        &= L\left|\int_{0}^t f^{\prime}(h(x-y(x,t), t)e^{\beta(s)})ds\right|
  \end{aligned}
     \end{equation*}
This implies that  $\|R(.,t)-R_0\|_{L^{\infty}(\mathbb{R})} \to 0$  as $t\to 0.$
  \end{proof}
\newpage
\begin{theorem} Consider $u$ as in \eqref{newdefinition of u}. Then, under the asumptions
\begin{itemize}
    \item  $\rho_0$ are bounded measurable functions
    \item $ R\in L^{\infty}_{loc}((0, \infty), BV_{loc}(\mathbb{R}, \mathbb{R}))$ 
    \item For any $\phi \in C_c ^{\infty} (\mathbb{R} \times (0,\infty))$, the map $t\to \int R\phi dx$ is continuous,
\end{itemize}
a weak solution in the sense of the definition \eqref{weak formulation of R} is unique.
\end{theorem}

\begin{proof}
{\bf Step 1:}    Recall that for all $\phi \in C_c^{\infty}(\mathbb{R}\times [0,\infty))$ the function R satisfies the weak formulation
 \begin{equation*}
 \int_0^\infty\int_{-\infty}^\infty R\phi_tdxdt- \int_0^\infty\int_{-\infty}^\infty f'(u(x,t))\phi(x,t)dR(x)dt=0.
\end{equation*}


Define $$ h_\epsilon(t)=\int_{0}^{\frac{t-s}{\epsilon}}\eta(z)dz,$$
where $\eta \in C_c^{\infty} (0, \infty)$. Then $h_\epsilon'(t)=\frac{1}{\epsilon}\eta(\frac{t-s}{\epsilon})$. \\
       
Now, choose the test function $\psi\in C_c^\infty (\mathbb{R}\times(0,\infty))$ defined as $\psi(x,t):=\phi(x,t)h_\epsilon(t)
$. Then we have
 \begin{equation*}
 \int_0^\infty\int_{-\infty}^\infty R\psi_tdxdt- \int_0^\infty\int_{-\infty}^\infty f'(u(x,t))\psi(x,t)dR(x)dt=0.
\end{equation*}
That is,
\begin{equation*}
 \int_0^\infty\int_{-\infty}^\infty R\phi_th_\epsilon(t)dxdt+\int_0^\infty\int_{-\infty}^\infty R\phi h_\epsilon' (t)dxdt- \int_0^\infty\int_{-\infty}^\infty f'(u(x,t))\phi(x,t)h_\epsilon(t)dR(x)dt=0.
\end{equation*}
Let us define $g:[0,\infty)\to \mathbb{R}$ as $g(t):=\int_{-\infty}^\infty R(x,t)\phi(x,t)dx$. Then the above equation becomes the following.
\begin{equation*}
 \int_0^\infty\int_{-\infty}^\infty R\phi_th_\epsilon(t)dxdt+\int_0^\infty\int_{-\infty}^\infty g(t) h_\epsilon' (t)dxdt- \int_0^\infty\int_{-\infty}^\infty f'(u(x,t))\phi(x,t)h_\epsilon(t)dR(x)dt=0.
\end{equation*}
Using the second and third assumptions,  passing to the limit as $\epsilon \to 0$, we obtain
\begin{equation*}
\begin{aligned}
 \int_0^\infty\int_{-\infty}^\infty R\phi_tH(t-s)dxdt&=-\lim_{\epsilon\to 0}\int_0^\infty\int_{-\infty}^\infty R\phi h_\epsilon' (t)dxdt+ \int_0^\infty\int_{-\infty}^\infty f'(u(x,t))\phi(x,t)h_\epsilon(t)dR(x,t)dt \\
 &=-\lim_{\epsilon\to 0}\frac{1}{\epsilon} \int_0^\infty g(t) \eta (\frac{t-s}{\epsilon})dt+ \int_s^\infty\int_{-\infty}^\infty f'(u(x,t))\phi(x,t)dR(x)dt\\
 &=-\lim_{\epsilon\to 0}\int_{-\frac{s}{\epsilon}}^\infty g(s+\epsilon z) \eta (z)dz+ \int_s^\infty\int_{-\infty}^\infty f'(u(x,t))\phi(x,t)dR(x,t)dt\\
 &=-g(s)+\int_s^\infty\int_{-\infty}^\infty f'(u(x,t))\phi(x,t)dR(x,t)dt\\
 &=-\int_{-\infty}^\infty R(x,s)\phi(x,s)dx +\int_s^\infty\int_{-\infty}^\infty f'(u(x,t))\phi(x,t)dR(x,t)dt.
\end{aligned}
\end{equation*}
Now, we have
 \begin{equation*}
 \int_s^\infty\int_{-\infty}^\infty R\phi_tdxdt+\int_{-\infty}^{\infty}R(x,s)\phi(x,s)dx- \int_s^\infty\int_{-\infty}^\infty f'(u(x,t))\phi(x,t)dR(x)dt=0.
\end{equation*}
Now in perticular, if we choose $\phi_x$ for some $\phi \in C_c^\infty(\mathbb{R}\times[0,\infty))$ then we have 
 \begin{equation*}
 \int_s^\infty\int_{-\infty}^\infty R\phi_{xt}dxdt+\int_{-\infty}^{\infty}R(x,s)\phi_x(x,s)dx =\int_s^\infty\int_{-\infty}^\infty f'(u(x,t))\phi_x(x,t)dR(x)dt.   
\end{equation*}
That is,
\begin{equation}
    \int_s^\infty \int_{-\infty}^\infty(\phi_t+f'(u)\phi_x)dRdt=\int_{-\infty}^\infty \phi_x(x,s)R(x,s)dx.
\end{equation}
Suppose $R_1$ and $R_2$ are two solutions to the PDE with the same initial condition $R_i(x,0)=R(x,0)=\int_0^x\rho_0(y)dy$. Since they both satisfy the above equation, for $\bar R:=R_1-R_2$ we have 
\begin{equation}
\label{4.22}
 \int_s^\infty \int_{-\infty}^\infty (\phi_t+f'(u)\phi_x)d\bar Rdt=\int_{-\infty}^\infty \phi_x(x,s)(R_1(x,s)-R_2(x,s))dx .    
\end{equation}

Let us define $f_\epsilon:\mathbb{R}\times[0,\infty)\to \mathbb{R}$ as $\int_0^1 f'(su_\epsilon(x+\epsilon^{\frac{1}{2}},t)+(1-s)u_\epsilon(x-\epsilon^{\frac{1}{2}},t)ds$.

 {\bf Step 2:} In this step, we prove that $f_\epsilon(x,t)\to f'(u(x,t))$ pointwise in $\mathbb{R}\times [0,\infty)$. Observe the following general fact. Let \(w\) be a bounded function such that the right limit \(w(x+)\) exists, and let \(\eta \in C_c^\infty(\mathbb{R})\) be a Friedrichs' mollifier. Define
\[
w_\epsilon(x)= (w*\eta_\epsilon)(x+\epsilon^{1/2}).
\]
Then
\[
\begin{aligned}
w_\epsilon(x)
&=\int_{-\infty}^{\infty} w(x+\epsilon^{1/2}-y)\,\eta_\epsilon(y)\,dy \\
&=\frac{1}{\epsilon}\int_{-\infty}^{\infty} w(x+\epsilon^{1/2}-y)\,
\eta\!\left(\frac{y}{\epsilon}\right)dy \\
&=\int_{-\infty}^{\infty} w(x-\epsilon z+\epsilon^{1/2})\,\eta(z)\,dz .
\end{aligned}
\]
Hence, as \(\epsilon \to 0\),
\[
\lim_{\epsilon\to 0} w_\epsilon(x)=w(x+).
\]
Similarly, if \(w\) is bounded and the left limit \(w(x-)\) exists, then one analogously obtains
\[
\lim_{\epsilon\to 0} w_\epsilon(x)=w(x-).
\]
Suppose $u$ is continuous at $(x,t)$, using the above general fact, 
\begin{equation*}
    \begin{aligned}
        \lim_{\epsilon\to 0}f_\epsilon(x,t)&=\lim_{\epsilon\to 0}\int_0^1 f'(su_\epsilon(x+\epsilon^{\frac{1}{2}},t)+(1-s)u_\epsilon(x-\epsilon^{\frac{1}{2}},t)ds\\
        &=\int_0^1 f'(su(x+,t)+(1-s)u((x-,t))ds=f'(u(x,t).
        \end{aligned}
        \end{equation*}
 If $(x,t)$ is a point of discontinuity, then
        \begin{equation*}
    \begin{aligned}
       &\lim_{\epsilon\to 0}f_\epsilon(x,t)\\&=\frac{1}{u(x+,t)-u(x-,t)}\int_0^1 f'(su(x+,t)+(1-s)u((x-,t))(u(x+,t)-u(x-,t))ds\\
        &=\frac{1}{u(x+,t)-u(x-,t)}\int_0^1 \frac{d}{ds}f(su(x+,t)+(1-s)u(x-,t))ds\\
        &=\frac{f(u(x+,t))-f(u(x-,t))}{u(x+,t)-u(x-,t)}.
    \end{aligned}
\end{equation*}

{\bf Step 3:} Let us choose a test function $\psi \in C_c^{\infty}(\mathbb{R}\times(0,T))$. Consider the following equation
\begin{equation}
\label{4.23}
    \begin{aligned}
        \phi_t^\epsilon+&f_\epsilon\phi_x^\epsilon =\psi \text{ in } \mathbb{R}\times[0,T]\\
        &\phi^\epsilon(x,T)=0.
    \end{aligned}
\end{equation}
We solve the above equation\eqref{4.23} using the method of charecteristics.  Let $x^\epsilon$ be the characteristic curve that satisfies
\begin{equation*}
    \begin{aligned}
        &\dot x^\epsilon(s)=f_{\epsilon}(x^\epsilon(s),s)\\
    &x^\epsilon(t)=x.
    \end{aligned}
\end{equation*}
Then the solution $\phi^\epsilon$ at the point $(x,t)$ is given by $$\phi^\epsilon(x,t)=-\int_t^T\psi(x^\epsilon(s),s)ds.$$
Since \(\psi\) is compactly supported, there exists \(0<\tau_1<\tau_2<T\) such that
\[
\psi(x,t)=0 \qquad \text{for } t\notin [\tau_1,\tau_2].
\]
Consequently, \(\phi^\epsilon\) is also supported in \([\tau_1,\tau_2]\) with respect to the time variable. Moreover, assume that \(\psi\) is supported in the spatial interval \([-M,M]\).
Since  $f_\epsilon$ is bounded, the characteristic equation
$|\dot x^\epsilon(s)|= |f_\epsilon(x^\epsilon(s),s)|$ implies that the slope of the characteristic curve is bounded. So, outside a compact set in $\mathbb{R}$, the function $\phi_\epsilon$ vanishes with respect to the $x$ variable as well. So, we can conclude that the function $\phi_\epsilon$ satisfying \eqref{4.23} is in $C_c^\infty(\mathbb{R}\times [0,\infty))$.\\ 

\noindent
{\bf Step 4:} In this step, we prove that given $s>0$, there exists a constant $C_s$ such that $|\phi^\epsilon_x(x,t)|\leq C_s $ for $s\leq t\leq T$.
For $(x,t)\in \mathbb{R}\times(s,T)$ we have 
\begin{equation*}
    \begin{aligned}
        &\dot x^\epsilon(s;x)=f_{\epsilon}(x^\epsilon(s;x),s)\\
    &x^\epsilon(t;x)=x.
    \end{aligned}
\end{equation*}
Now, differentiating $\dot x^\epsilon(.,x)$ with respect to x we get 
\begin{equation}
    \begin{aligned}
        &\dot x^\epsilon_x(s;x)= f_{\epsilon x}(x^\epsilon(s;x),s) x^\epsilon_x(s;x)\\
        & x^\epsilon_x(t;x)=1.
    \end{aligned}
\end{equation}
Solving this ODE, we get 
\begin{equation*}
   x^\epsilon_x(s;x)=1+\int_t^s f_{\epsilon x}(x^\epsilon(\tau;x),\tau)x^\epsilon_x(\tau;x)d\tau.
\end{equation*}
Applying Gronwall's inequality, we get
\begin{equation*}
    |x^\epsilon_x(s;x)|\leq \exp\left(\int_t^sf_{\epsilon x}(x^\epsilon(\tau;x),\tau)d\tau\right).
\end{equation*}
Now $$f_{\epsilon x}(x,t)=\left(\int_0^1f''(su_\epsilon(x+\epsilon^{\frac{1}{2}},t)+(1-s)u_\epsilon(x-\epsilon^{\frac{1}{2}},t))ds\right)(su_{\epsilon x}(x+\epsilon^{\frac{1}{2}},t)+(1-s)u_{\epsilon x}(x-\epsilon^{\frac{1}{2}},t))$$
Now due to \cite[Theorem 2.6(vii)]{manish}(one side Lipschitz property) we see that for $z>0$ and $t\in [s,T]$,
\begin{equation*}
\begin{aligned}
    \frac{u_\epsilon(x+z,t)-u_\epsilon(x,t)}{z}&=\int_{-\infty}^\infty \frac{u(x-y+z)-u(x-y)}{z}\eta_\epsilon(y)dy\\
    &\leq \frac{C}{\int_0^te^{\beta(\tau)-\beta(t)}d\tau}\leq CM_s:=D_s\\
    &\left(\textnormal{where } M_s=\frac{1}{\inf_{\tau\in [s,T]}\int_0^te^{\beta(\tau)-\beta(t)}d\tau}\right)
\end{aligned}
\end{equation*}
Thus, taking $z\to 0$ we get $|u_x^\epsilon(x,t)|\leq D_s$ for all $(x,t)\in \mathbb{R}\times[s,T]$. Consequently $|f_{\epsilon x}(x,t)|\leq C_s$ for some constant $C_s$ depending on s.\\
 Now, differentiating $\phi^\epsilon(x,t)$ with respect to $x$ we get
 \begin{equation*}
 \begin{aligned}
 \phi_x^\epsilon(x,t)&=-\int_t^T\psi_x(x^\epsilon(s),s) x^\epsilon_x(s;x)ds\\
 &=-\int_t^T\psi_x(x^\epsilon(s),s)f_x^{\epsilon}(x^\epsilon(s;x),s)ds.
 \end{aligned}
 \end{equation*}
 Hence, for $s<t\leq T$,
 \begin{equation*}
 |\phi_x^\epsilon(x,t)|\leq ||\psi_x||_{L_\infty}(T-t)D_s\leq ||\psi_x||_{L_\infty}D_sT  :=C_s . 
 \end{equation*}
 
 {\bf Step 5:} Using \eqref{4.22} and \eqref{4.23} now we can write
 \begin{equation}
 \label{4.25}
     \begin{aligned}
         \int_s^T\int_{-\infty}^\infty \psi(x,t)d\bar R dt&=\int_s^T\int_{-\infty}^\infty (\phi_t^\epsilon+f_\epsilon\phi_x^\epsilon) d\bar R dt\\
         &=\int_s^T\int_{-\infty}^\infty (f_\epsilon-f'(u))\phi_x^\epsilon) d\bar R dt+\int_s^T\int_{-\infty}^\infty (\phi_t^\epsilon+f'(u)\phi_x^\epsilon) d\bar R dt\\
         &=\int_s^T\int_{-\infty}^\infty (f_\epsilon-f'(u))\phi_x^\epsilon) d\bar R dt+\int_{-\infty}^\infty \phi_x(x,s)(R_1(x,s)-R_2(x,s))dx
     \end{aligned}
 \end{equation}
 Now
\begin{equation*}
    \begin{aligned}
     \int_{-\infty}^\infty \phi_x(x,s)(R_1(x,s)-R_2(x,s))dx& \leq \int_{-\infty}^\infty| \phi_x(x,s)(R_1(x,s)-R_2(x,s))|dx\\
     & \leq M||R_1(.,s)-R_2(.,s)||_{L^\infty}\\
     &\leq M(||R_1(.,s)-R_1(.,0)||_{L^\infty}+||R_2(.,s)-R_2(.,0)||_{L^\infty})\to0 \text{ as } s\to 0.
    \end{aligned}
\end{equation*}
So the second term of the right hand side of the equation \eqref{4.25} vanishes.
 Let us denote  $I_\epsilon:=\int_s^T\int_{-\infty}^\infty (f_\epsilon-f'(u))\phi_x^\epsilon) d\bar R dt$ which is the remaining term.\\
Then since from step 4 for $[s,T]$,  $|\phi_x^\epsilon(x,t)|\leq C_s$ and $f_\epsilon\to f'(u)$ almost everywhere and also since $\phi ^\epsilon_x$ is compactly supported, we can apply DCT and conclude that $I_\epsilon\to 0$ as $\epsilon\to 0$.\\

Consequently, taking $s\to 0$ and $\epsilon\to 0$  we finally get
\begin{equation*}
    \int_0^T\int_{-\infty}^\infty \psi(x,t)d\bar R dt=0 \textnormal{ for all } \psi\in C_c^\infty(\mathbb{R}\times (0,T).
\end{equation*}
Since this happens for all $T>0$, we can conclude that $\bar R$ is zero almost everywhere. So the solution $R$ is unique.

\end{proof}

\begin{remark}{\bf Uniqueness of Solution for system \eqref{equation 1}:}
Suppose $\rho_1, \rho_2$ are two solutions for the system \eqref{equation 1} with initial conditions \eqref{equation 2}. Then we can find $R_1, R_2$ that satisfy $\partial_xR_1=\rho_1$ and $\partial_xR_2=\rho_2$. But $R_1, R_2$ both satisfy the system \eqref{equation for R} with initial condition \eqref{initial condition for u and R}. From the above theorem we have $R_1=R_2$. Consequently $\rho_1=\rho_2$. So, the solution of \eqref{equation 1} exists and is unique.
\end{remark}

\bibliographystyle{plain}	
\bibliography{Reference}
\end{document}